\documentclass[pdflatex]{sn-jnl}
\usepackage{amsfonts,amsmath,amssymb}
\usepackage{graphicx} 

\usepackage{url}
\usepackage{amsthm}
\usepackage{stmaryrd}
\usepackage{xcolor}
\usepackage[normalem]{ulem}
\usepackage{daktikz}
\usepackage{subfig} 
\usetikzlibrary{patterns}
\usepackage[T1]{fontenc}
\usepackage{lmodern}

\usepackage[bordercolor=white,backgroundcolor=gray!30,linecolor=black,colorinlistoftodos]{todonotes}

\def\ker{\mathop{\rm Ker}\nolimits}

\newcommand{\CC}{\mathbb C}

\newtheorem{theorem}{Theorem}[section]

\newtheorem{example}[theorem]{Example}

\newtheorem{lemma}[theorem]{Lemma}

\DeclareMathAlphabet{\mathup}{OT1}{\familydefault}{m}{n}
\newcommand*{\mup}[1]{\mathup{#1}} 

\newcommand*{\iu}{\mathup{i}\mkern1mu}
\newcommand\ten[1]{\mathbf{#1}}
\newcommand\dirvec[1]{\mathbf{e}_{#1}}
\DeclareMathOperator{\e}{e}

\makeatletter
\renewcommand*\env@matrix[1][c]{\hskip -\arraycolsep
  \let\@ifnextchar\new@ifnextchar
  \array{*\c@MaxMatrixCols #1}}
\makeatother

\begin{document}

\title{A Sylvester equation approach for the computation of zero-group-velocity points in waveguides}
\author*[1,2]{Bor Plestenjak}\email{bor.plestenjak@fmf.uni-lj.si}
\author[3]{Daniel A. Kiefer}\email{daniel.kiefer@espci.fr}
\author[4]{Hauke Gravenkamp}\email{gravenkamp.research@gmail.com}

\affil*[2]{\orgdiv{Faculty of Mathematics and Physics}, \orgname{University of Ljubljana}, \orgaddress{\street{Jadranska 19}, \city{1000 Ljubljana}, \country{Slovenia}}}
\affil*[3]{\orgname{Institute of Mathematics, Physics and Mechanics}, \orgaddress{\street{Jadranska 19}, \city{1000 Ljubljana}, 
\country{Slovenia}}}

\affil[3]{\orgdiv{Institut Langevin, ESPCI Paris}, \orgname{Universit\'e PSL, CNRS}, \orgaddress{\city{75005 Paris}, \country{France}}}
\affil[4]{\orgname{International Centre for Numerical Methods in Engineering (CIMNE)}, \orgaddress{\city{08034 Barcelona}, \country{Spain}}}

\hyphenation{wave-gui-de wave-guides}

\graphicspath{{img/},{tikz/}}

\abstract{ 
    Eigenvalues of parameter-dependent quadratic eigenvalue problems form eigencurves. The critical points on these curves, where the derivative vanishes, are of practical interest. 
    A particular example is found in the dispersion curves of elastic waveguides, where such points are called zero-group-velocity (ZGV) points. Recently, it was revealed that the problem of computing ZGV points can be modeled as a multiparameter eigenvalue problem (MEP), and several numerical methods were devised. Due to their complexity, these methods are feasible only for problems involving small matrices. In this paper, we improve the efficiency of these methods by exploiting the link to the Sylvester equation. This approach enables the computation of ZGV points for problems with much larger matrices, such as multi-layered plates and three-dimensional structures of complex cross-sections.}


\pacs[Funding]{
Bor Plestenjak has been supported by the Slovenian Research and Innovation Agency (grants N1-0154 and P1-0194). Daniel A.\ Kiefer has received support under the program ``Investissements d'Avenir'' launched by the French Government under Reference No.\ ANR-10-LABX-24.}

\maketitle

\section{Introduction}

In many physics and engineering applications, we encounter parameter-dependent quadratic eigenvalue problems (QEP) of the form
\begin{equation}\label{eq:ZGV1}
    W(k,\omega)u:=\big((\iu k)^2 L_2+ \iu k L_1+L_0+\omega^2 M\big)\,u=0,
\end{equation}
where $L_0$, $L_1$, $L_2$, $M$ are real $n\times n$ matrices, which are usually obtained by a (semi-)discretization
of a boundary value problem. The solutions $(k,\omega)$ form \emph{eigencurves} $\omega(k)$, and we are interested in locating the critical points on these curves, where $\omega'(k)= \frac{\partial \omega}{\partial k} = 0$. Although solutions of \eqref{eq:ZGV1} can be complex, we consider the important case where $\omega$ and $k$ are both real.

This work is motivated by the study of (anisotropic) elastic waveguides (see, e.g., \cite{ZGV_JASA_23, prada_local_2008}), where $\omega$ denotes the angular frequency and $k$ the wavenumber. In this context, the eigencurves are referred to as \emph{dispersion curves}. The slope $c_\mup{g} = \omega'$ is called \emph{group velocity}, which is of practical relevance, as it describes the propagation of energy. Points $(k_*, \omega_*)$ on the dispersion curves where the group velocity vanishes are called \emph{zero-group-velocity (ZGV)} points. Often, the term is used exclusively for solutions at finite wavenumber~$k_*$, as this is the non-trivial case, but here we use the designation for solutions at any $k_*$.
 In the light of this motivating practical application, we will generally refer to points on the curves formed by eigenvalues of parameter-dependent eigenvalue problems that satisfy $\omega'(k)=0$ as \emph{ZGV points}, irrespective of their physical interpretation.

Recently, a numerical algorithm for the computation of ZGV points in anisotropic elastic waveguides was introduced \cite{ZGV_JASA_23} that can be applied to a general problem of the form \eqref{eq:ZGV1}. The method is based on a
generalization of the method of fixed relative distance (MFRD) from \cite{Elias_DoubleEig}, which provides good initial approximations that can be refined by a locally convergent Newton-type method. Inspired by the Sylvester-Arnoldi method from \cite{MP_SylvArnoldi}, we show in this contribution that sophisticated tools from linear algebra substantially speed up the algorithm and reduce its memory requirements. This enables us to solve problems with larger matrices and tackle
more complex problems such as multi-layered plates as well as waveguides of arbitrary two-dimensional cross-sections.

In the following, we first discuss properties of ZGV points in Section~\ref{sec:theory}. In Section~\ref{sec:aux}, we introduce several tools we will use in the following section; the presentation is intertwined with their application to the computation of ZGV points: the Sylvester equation, multiparameter eigenvalue problems, and the MFRD. Our main contributions are included in Section~\ref{sec:impr}, where we show how we can exploit the structure of the Sylvester equation to apply the MFRD more efficiently, and in Section~\ref{sec:alg}, where we present a scanning algorithm for the computation of ZGV points that combines the MFRD and a locally convergent Gauss-Newton method. In Section~\ref{sec:waveguide}, we introduce a waveguide model that is used in the numerical experiments in the following section, where we demonstrate the strength of the proposed method. Finally, we discuss possible generalizations and give a conclusion in Sections~\ref{sec:generalization} and \ref{sec:conclusion}.

\section{Theory on ZGV points}\label{sec:theory}

If we assume that $u=u(k)$ and $\omega=\omega(k)$ are differentiable, then, by differentiating \eqref{eq:ZGV1}, 
we obtain
\[
    (-2k L_2 +\iu L_1 + 2\omega(k)\omega'(k)M)u(k) + W(k,\omega(k))u'(k)=0,
\]
where, in general, $\bullet' = \frac{\partial \bullet}{\partial k}$. If $(k_*, \omega_*)$ is a ZGV point, then $\omega'(k_*)=0$, and it follows that
\begin{equation}\label{eq:dif_ZGV}
    (-2k_* L_2 +\iu L_1)u_* + W(k_*,\omega_*)v_*=0,
\end{equation}
where $u_*=u(k_*)$ and $v_*=u'(k_*)$.
Let $z_*$ be the corresponding left eigenvector of \eqref{eq:ZGV1} at $(k_*,\omega_*)$, i.e., $z_*^\mup{H} W(k_*,\omega_*)=0$. By multiplying \eqref{eq:dif_ZGV} by $z_*^\mup{H}$ from the left, we get the following necessary condition for a ZGV point:
\begin{equation}\label{eq:cond_ZGV}
    z_*^\mup{H}(-2k_* L_2 +\iu L_1 )u_*= z_*^\mup{H} W'(k_*,\omega_*)u_*=0.
\end{equation}

\begin{lemma}\label{lem:ZGV_multiple_eig}
    If $(k_*,\omega_*)$ is a ZGV point of the parameter-dependent QEP \eqref{eq:ZGV1}, then $k_*$ is a multiple eigenvalue of the
    QEP
    \begin{equation}\label{eq:QEP_Q}
        Q(k)u:=\big((\iu k)^2 L_2+ \iu k L_1+L_0+\omega_*^2 M\big)\,u=0
    \end{equation}
    that we get by fixing $\omega$ in \eqref{eq:ZGV1} to $\omega_*$.
\end{lemma}
\begin{proof}
    If $z$ and $u$ are the left and right eigenvector of a simple eigenvalue
    $k$ of the QEP $Q(k)u = 0$, then it is well-known that
    $z^\mup{H}Q'(k)u\ne 0$, see, e.g., \cite[Prop. 1]{NeumaierSIMAX} or \cite[Thm. 3.2]{AndrewChuLancaster_Derivatives}.
    But, since \eqref{eq:cond_ZGV} holds at a ZGV point, it thus follows that $k_*$ is a multiple eigenvalue of \eqref{eq:QEP_Q}.
\end{proof}

Lemma \ref{lem:ZGV_multiple_eig} gives a necessary condition, but not every point $(k_*,\omega_*)$ such that $k_*$ is a multiple eigenvalue of \eqref{eq:QEP_Q} for a fixed $\omega=\omega_*$, corresponds to a ZGV point. In addition, \eqref{eq:dif_ZGV}
must hold as well, and this means that $v_*$ is a root vector of height two\footnote{An eigenvalue $\lambda_*$ of a quadratic matrix polynomial 
$Q(\lambda)$ is defective if and only if there exist an eigenvector (a root vector of height one) $v\ne 0$ and a root vector of height two $w$ such that
$Q(\lambda_*)v=0$ and $Q(\lambda_*)w+Q'(\lambda_*)v=0$, see, e.g., \cite[Sec.~2.2]{GuttelTisseur_NEP}.}. 
This is possible only if the algebraic multiplicity $m_a$ of $k_*$ as an eigenvalue of \eqref{eq:QEP_Q} is strictly greater than the geometric multiplicity $m_g=\dim(\ker(Q(k_*)))$. Also, to make sure that $\omega(k)$ is analytic in a neighborhood of $k_*$, we require that $\omega_*$ is a simple eigenvalue of $W(k_*,\omega)$, i.e., the generalized eigenvalue problem (GEP) that we get by fixing $k$ to $k_*$ in \eqref{eq:ZGV1}. Note that in some cases, it is possible to extend the dispersion curves so that they remain analytical also in points where the curves cross and multiple eigenvalues appear, see, e.g., \cite{LuSuBai_SIMAX}. To keep things concise, we will keep the requirement that $\omega_*$ is a simple eigenvalue of $W(k_*,\omega)$ and thus exclude points from candidates for ZGV points where two or more dispersion curves cross.

\begin{example}\rm\label{ex:introduction}
    We take
    \[
        L_2=\left[\begin{matrix}2 & 1  &0\cr 1 & 1 & 0\cr 0 & 0 & 1\end{matrix}\right],\
        L_1=\begin{bmatrix}[r]0 & 3  &0\cr -3 & 0 & 0\cr 0 & 0 & 0\end{bmatrix},\
        L_0=\left[\begin{matrix}-1.75 & 1  & 0\cr 1 & -1.75 & 0\cr 0 & 0 & -0.25\end{matrix}\right],\
        M=\left[\begin{matrix}3 & 1  &0\cr 1 & 4 & 0\cr 0 & 0 & 3.5\end{matrix}\right].
    \]
    We selected the matrices so that $L_0$ is symmetric, $L_1$ is skew-symmetric, and $L_2, M$ are symmetric positive definite. This way, the matrices have the same properties as the larger matrices in \cite{ZGV_JASA_23}, where ZGV points of Lamb waves in an austenitic steel plate are computed.

    The corresponding eigenvalue problem \eqref{eq:ZGV1} has five real ZGV points $(0,0.2673)$, $(0,0.4074)$, $(0,1.0628)$, and $(\pm 1.0642,0.2393)$,
    such that $\omega>0$, which are shown together with the real dispersion curves in Fig.~\ref{fig:introduction}. We consider only the solutions 
    with $\omega>0$ since each dispersion curve $\omega(k)$ has its counterpart $-\omega(k)$ and the same holds for the ZGV points.
    Note that the points $(\pm 0.4236,0.3503)$, where the dispersion curves cross, are not ZGV points, although $k_*=\pm 0.4236$ is a double eigenvalue of \eqref{eq:QEP_Q} for a fixed $w_*=0.3503$. 
        
    Due to the structure of the matrices, the dispersion curves are also symmetric with respect to the $\omega$-axis and there exist trivial ZGV points at $k=0$, which can be computed 
    from the GEP $(L_0+\omega^2 M)u=0$. Nontrivial ZGV points come in pairs $(\pm k_*, \omega_*)$ and we are interested
    in solutions where $k_*>0$.
    \begin{figure}[htb!]
        \centering
        \includegraphics[height=6cm]{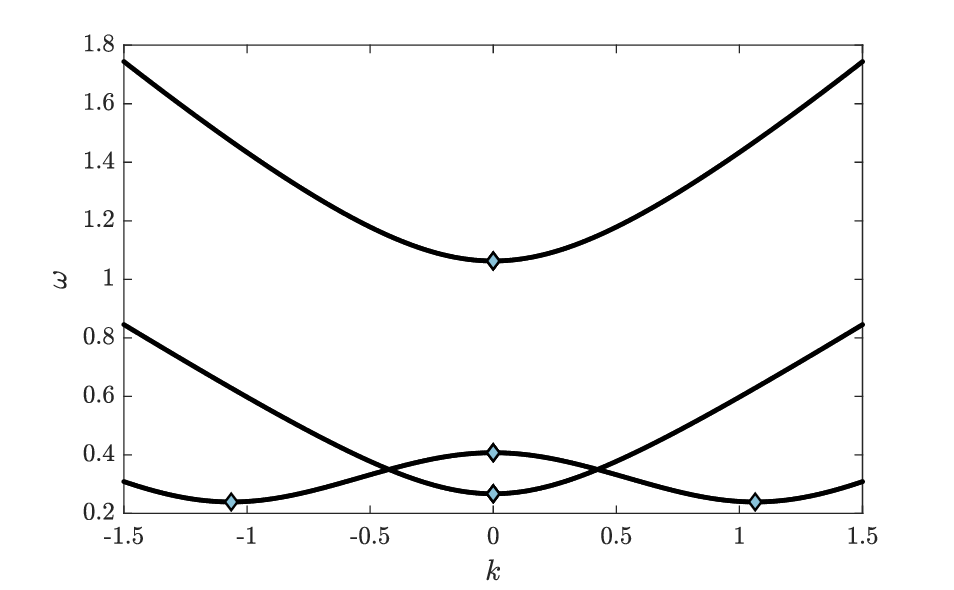}
        \caption{Real dispersion curves $\omega(k)$ and ZGV points of Example~\ref{ex:introduction}.}
        \label{fig:introduction}
    \end{figure}
\end{example}

\section{Auxiliary results}\label{sec:aux}

In this section, we introduce some related results and numerical methods that we will use in the following sections
to construct an efficient numerical method for finding the ZGV points of \eqref{eq:ZGV1}.

\subsection{Sylvester equation}\label{sub:Sylvester}
\emph{The Kronecker product} $A\otimes B$ of matrices $A\in\CC^{n\times m}$ and
$B\in\CC^{p\times q}$ is
a matrix of size $np\times mq$ of the block form
\[A\otimes B = \left[\begin{matrix}
            a_{11}B & \cdots & a_{1m}B\cr
            \vdots  &        & \vdots\cr
            a_{n1}B & \cdots & a_{nm}B
        \end{matrix}\right].
\]
For a matrix $X\in\CC^{m\times n}$, ${\rm vec}(X)\in\CC^{mn}$ is a \emph{vectorization} of matrix $X$, i.e, the vector obtained by stacking all columns of $X$ on top of each other. Our results are based on the well-known equality (see, e.g., \cite[Lem.~4.3.1]{Horn_Johnson_Topics_1991}):
\begin{equation}\label{eq:AXB}
    {\rm vec}(AXB) = (B^\mup{T}\otimes A)\,{\rm vec}(X),
\end{equation}
which holds for $A\in\CC^{m\times m}$, $B\in\CC^{n\times n}$, and $X\in\CC^{m\times n}$.
Suppose that we are looking for a matrix $X\in\CC^{m\times n}$ that satisfies \emph{the Sylvester equation}
\begin{equation}\label{eq:sylvester}
    AX+XB=C
\end{equation}
for given matrices $A\in\CC^{m\times m}$, $B\in\CC^{n\times n}$, and $C\in\CC^{m\times n}$.
It follows from \eqref{eq:AXB} that \eqref{eq:sylvester} is equivalent to
\begin{equation}\label{eq:sylvester_kronecker}
    (I_n\otimes A + B^\mup{T}\otimes I_m){\rm vec}(X)={\rm vec}(C).
\end{equation}
The Sylvester equation is therefore uniquely solvable when $I_n\otimes A + B^\mup{T}\otimes I_m$ is nonsingular, which is true if and only if $\lambda + \mu\ne 0$ for all possible pairs $(\lambda,\mu)$, where $\lambda$ is an eigenvalue of $A$ and $\mu$ is an eigenvalue of $B$, see, e.g., \cite[Thm.~4.4.6]{Horn_Johnson_Topics_1991}.
We could apply \eqref{eq:sylvester_kronecker} to numerically solve \eqref{eq:sylvester}, but this is not efficient since it, in general, leads to complexity ${\cal O}(m^3n^3)$ due to a matrix of size $mn\times mn$ in \eqref{eq:sylvester_kronecker}.

There exist more efficient numerical methods for the Sylvester equation, for instance, the Bartels-Stewart algorithm \cite{Bartels_Stewart}, which is appropriate in our setting, where we have to solve many Sylvester equations with the same matrices $A,B$ and different right-hand sides $C$. In the Bartels-Stewart algorithm, we first compute two Schur decompositions \[A=QRQ^\mup{H},\quad B=USU^\mup{H},\]
where matrices $Q,U$ are unitary, and matrices $S,R$ are upper triangular. Applying the above Schur decompositions to \eqref{eq:sylvester}, we obtain a new Sylvester equation with upper triangular matrices
\begin{equation}\label{eq:sylvesterRS}
    RY+YS=D,
\end{equation}
where $D=Q^\mup{H}CU$ and $Y=Q^\mup{H}XU$.
The columns of $Y=\left[\begin{matrix}y_1 & \cdots & y_n\end{matrix}\right]$ can now be computed from left to right as solutions of upper triangular linear systems
\[
    (R+s_{ii}I)y_i = d_i-\sum_{k=1}^{i-1}s_{ki}y_k,\quad i=1,\ldots,n,
\]
and then $X=QYU^\mup{H}$. If the Sylvester equation is nonsingular, then $R+s_{ii}I$ is nonsingular for all $i=1,\ldots,n$. With the above approach, we can efficiently solve the Sylvester equation \eqref{eq:sylvester} in complexity
${\cal O}(m^3+n^3)$, which is much less than ${\cal O}(m^3n^3)$, the complexity of solving \eqref{eq:sylvester_kronecker} as a large linear system.

\subsection{Multiparameter eigenvalue problems}\label{sec:mep}

A \emph{$d$-parameter eigenvalue problem} has the form
\begin{align}
    A_{10}x_1 & = \lambda_1 A_{11}x_1 + \cdots + \lambda_d A_{1d}x_1\nonumber  \\
              & \vdots \label{eq:mep}                                          \\
    A_{d0}x_d & = \lambda_1 A_{d1}x_d + \cdots + \lambda_d A_{dd}x_d,\nonumber
\end{align}
where $A_{ij}$ is an $n_i\times n_i$ complex matrix, and $x_i\ne 0$  are vectors for $i=1,\ldots,d$.
If \eqref{eq:mep} holds, then $(\lambda_1,\ldots,\lambda_d)\in\CC^d$ is an \emph{eigenvalue} and $x_1\otimes \cdots \otimes x_d$ is the corresponding \emph{eigenvector}. A generic multiparameter eigenvalue problem (MEP) \eqref{eq:mep} has $N=n_1\cdots n_d$ eigenvalues, which are roots of a system of $d$ multivariate characteristic polynomials
\begin{equation}\label{eq:mep_charpoly}
    p_i( \lambda_1,\ldots,\lambda_d):=\det(A_{i0}-\lambda_1 A_{i1} -\cdots -\lambda_d A_{id})=0,\quad i=1,\ldots,d.
\end{equation}

The problem \eqref{eq:mep} is related to a system of GEPs
\begin{equation}\label{eq:Delta}
    \Delta_1 z =\lambda_1 \Delta_0z,\quad \ldots,\quad \Delta_d z=\lambda_d\Delta_0z,
\end{equation}
where $z=x_1\otimes\cdots\otimes x_d$, and the $N\times N$ matrices
\begin{equation}\label{eq:Operdet_Delta0}
    \Delta_0=\left|\begin{matrix}A_{11} & \cdots & A_{1d}\cr
              \vdots &        & \vdots \cr
              A_{d1} & \cdots & A_{dd}\end{matrix}\right|_\otimes
    :=\sum_{\sigma\in S_d}{\rm sgn}(\sigma) \, A_{1\sigma_1}\otimes A_{2\sigma_2}\otimes \cdots \otimes A_{d\sigma_d}
\end{equation}
\[\Delta_i=\left|\begin{matrix}A_{11} & \cdots & A_{1,i-1} & A_{10} & A_{1,i+1} & \cdots & A_{1d}\cr
              \vdots &        & \vdots    & \vdots & \vdots    &        & \vdots \cr
              A_{d1} & \cdots & A_{d,i-1} & A_{d0} & A_{d,i+1} & \cdots & A_{dd}\end{matrix}\right|_\otimes,\quad i=1,\ldots,d,
\]
are called \emph{operator determinants}, which uses the Kronecker product for multiplication. For details see, e.g., \cite{AtkinsonBook}. If $\Delta_0$ is nonsingular, then we say that \eqref{eq:mep} is \emph{regular}. In such cases, the matrices $\Gamma_i:=\Delta_0^{-1}\Delta_i$, $i=1,\ldots,d$, commute, and the eigenvalues of \eqref{eq:mep} are
the joint eigenvalues of commuting matrices $\Gamma_1,\ldots,\Gamma_d$. Hence, if $N$ is not too large, a standard numerical approach to computing the eigenvalues of \eqref{eq:mep} is to explicitly compute $\Gamma_1,\ldots,\Gamma_d$ and then solve a joint eigenvalue problem. Alternatively, if we prefer not to multiply by $\Delta_0^{-1}$, we may solve a joint system of GEPs \eqref{eq:Delta}, see, e.g., \cite{HKP_JD2EP}.

If all linear combinations of matrices $\Delta_0,\Delta_1,\ldots,\Delta_d$ are singular, then \eqref{eq:mep} is a \emph{singular MEP}, which is much more difficult to solve. In such case, it is still possible that the polynomial system \eqref{eq:mep_charpoly} has finitely many roots that are the eigenvalues of \eqref{eq:mep}. Then, \eqref{eq:Delta} is a
joint system of $d$ singular matrix pencils whose regular eigenvalues are the solutions of \eqref{eq:mep}. For more details, see, e.g., \cite{KP_SingGep2}.
To solve a singular MEP numerically, we can apply a generalized staircase-type algorithm \cite{MP_Q2EP}, which returns
matrices $Q$ and $Z$ with orthonormal columns that yield projected smaller matrices $\widehat\Delta_i=Q^*\Delta_i Z$ for $i=0,\ldots,d$ such that $\widehat\Delta_0$ is nonsingular,  matrices $\widehat\Delta_0^{-1}\widehat\Delta_i$, $i=1,\ldots,d$, commute, and their joint eigenvalues are the eigenvalues of \eqref{eq:mep}.

The above approach for singular problems is used in \cite{ZGV_JASA_23}, where it is shown that ZGV points of \eqref{eq:ZGV1} correspond to the eigenvalues of a singular three-parameter eigenvalue problem  (3EP)
\begin{align}
    (\eta C_2 + \lambda C_1 + C_0) w                                                    & =0\nonumber       \\
    (\eta L_2 + \lambda L_1 + L_0 + \mu M)u                                             & =0\label{eq:3ep2} \\
    (\eta \widetilde L_2 + \lambda \widetilde L_1 + \widetilde L_0 + \mu \widetilde M)v & =0,\nonumber
\end{align}
where $\lambda=\iu k$, $\mu=\omega^2$, $\eta=(\iu k)^2$,
\[
    \widetilde L_2 = \begin{bmatrix} L_2 & 0 \\ 0 & L_2 \end{bmatrix},\
    \widetilde L_1 = \begin{bmatrix} L_1 & 0 \\ 2L_2 & L_1 \end{bmatrix},\
    \widetilde L_0 = \begin{bmatrix} L_0 & 0 \\ L_1 & L_0 \end{bmatrix},\
    \widetilde M = \begin{bmatrix} M & 0 \\ 0 & M \end{bmatrix},
\]
and
\begin{equation}\label{eq:C0C1C2}
    C_2 = \begin{bmatrix} 1 & 0 \\ 0& 0 \end{bmatrix},\
    C_1 = \begin{bmatrix}[r] 0 & -1 \\ -1 & 0 \end{bmatrix},\
    C_0 = \begin{bmatrix} 0 & 0 \\ 0 & 1 \end{bmatrix}.
\end{equation}
Note that the matrices \eqref{eq:C0C1C2} incorporate the relation between $\lambda$ and $\eta$ since $\det(\eta C_2 + \lambda C_1 + C_0)=\eta-\lambda^2$. This gives the first known numerical method that can compute all ZGV points without any initial approximations. However, since we first have to explicitly compute the corresponding $\Delta$ matrices of size $4n^2\times 4n^2$ and then use expensive numerical methods for singular problems, this approach is feasible only for problems \eqref{eq:ZGV1} with small matrices.

Let us remark that solvers for MEPs are not included in standard numerical packages for problems in linear algebra, but several numerical methods, also for singular problems, are implemented in the Matlab toolbox MultiParEig \cite{multipareig_28}.

\subsection{Method of fixed relative distance\label{sec:mfrd}}
Instead of solving \eqref{eq:3ep2}, which gives exact solutions, we can solve a simpler regular 3EP that returns approximations of candidates for ZGV points. From each individual candidate, we can then compute the exact ZGV point by applying the locally convergent Gauss-Newton method that we provide in Section \ref{sec:Gauss-Newton}. This approach, presented first in \cite{ZGV_JASA_23} for Hermitian problems, is based on Lemma \ref{lem:ZGV_multiple_eig} and the method of fixed relative distance (MFRD) from \cite{Elias_DoubleEig}.

We know from Lemma \ref{lem:ZGV_multiple_eig} that at $\omega_*$ corresponding to a ZGV point, the QEP in the variable $\lambda = \iu k$ with fixed $\mu_* = \omega_*^2$, i.e.,
\begin{equation}
    \label{eq:waveguide_problem_discrete_pert}
    \left( \lambda^2 L_2 + \lambda L_1 + L_0 + \mu_* M \right) u = 0 \,,
\end{equation}
has a multiple (generically double) eigenvalue~$\lambda_* = \iu k_*$.
Therefore, for certain $\widetilde\mu \ne \mu_*$ but close to $\mu_*$, the QEP
\begin{equation}\label{eq:pert_QEP}
    \left(\lambda^2 L_2 + \lambda L_1 + L_0 + \widetilde\mu M \right) u=0
\end{equation}
has at least two different solutions close to $\lambda_*$. The MFRD, adapted to \eqref{eq:waveguide_problem_discrete_pert} in~\cite{ZGV_JASA_23}, introduces the 3EP
\begin{align}
    (\eta C_2 + \lambda C_1 + C_0) w                                           & =0\nonumber           \\
    (\eta L_2 + \lambda L_1 + L_0 + \mu M )u                                   & =0\label{eq:3ep_pert} \\
    \left(\eta (1+\delta)^2 L_2 + \lambda (1+\delta) L_1 + L_0 + \mu M\right)v & =0\nonumber
\end{align}
in $\mu = \omega^2$, $\lambda = \iu k$, $\eta = \lambda^2$, and $C_0,C_1,C_2$ as in \eqref{eq:C0C1C2}. Therein, $\delta>0$ specifies the relative distance between the sought $\lambda$ and serves as a regularization parameter.
The 3EP \eqref{eq:3ep_pert} has an eigenvalue $(\widetilde\lambda,\widetilde \mu,{\widetilde\eta})$, where
$\left(\sqrt{\widetilde\mu},-\iu\widetilde\lambda\right)$ is close to a ZGV point, such that $\widetilde\lambda$ and $\widetilde\lambda(1+\delta)$ are eigenvalues of the initial problem~\eqref{eq:pert_QEP}.

Solutions of \eqref{eq:3ep_pert} can be obtained from a transformation into the corresponding system of GEPs
\eqref{eq:Delta}. The problem \eqref{eq:3ep_pert} is regular since the corresponding $2n^2\times 2n^2$ matrix
\begin{equation}
    \Delta_0 = \left|\begin{matrix}
        C_2              & C_1           & 0 \cr
        L_2              & L_1           & M\cr
        (1+\delta)^2 L_2 & (1+\delta)L_1 & M
    \end{matrix}\right|_\otimes\label{eq:threeeig_pert_1}
\end{equation}
is nonsingular for $\delta>0$. Hence, we can solve the GEP given by
\begin{equation}\label{eq:Delta_pert_3par}
    \Delta_1z = \lambda \Delta_0 z,
\end{equation}
where
\begin{equation}
    \Delta_1 = (-1)\left|\begin{matrix}
        C_2              & C_0 & 0\cr
        L_2              & L_0 & M\cr
        (1+\delta)^2 L_2 & L_0 & M
    \end{matrix}\right|_\otimes,\label{eq:threeeig_pert_2}
\end{equation}
by a standard numerical method for GEPs. In the ensuing, the obtained eigenvector $z$ is used in the GEP associated with $\mu$, namely,
\begin{equation}\label{eq:Delta_pert_3par_mu}
    \Delta_M z = \mu \Delta_0 z,
\end{equation}
with $\Delta_0$ defined before and
\begin{equation}\label{eq:threeeig_pert_M}
    \Delta_M  = (-1) \left|\begin{matrix}
        C_2              & C_1           & C_0 \cr
        L_2              & L_1           & L_0 \cr
        (1+\delta)^2 L_2 & (1+\delta)L_1 & L_0
    \end{matrix}\right|_\otimes,
\end{equation}
to obtain $\mu$ from the Rayleigh quotient
\begin{equation}\label{eq:Rayleigh_mu}
    \mu = \frac{z^\mup{H} \Delta_M\, z}{z^\mup{H} \Delta_0\, z} \,.
\end{equation}
Even for small values of $n$, computing all eigenvalues of \eqref{eq:Delta_pert_3par} is very demanding. Instead, we can apply a subspace iterative method (for instance, {\tt eigs} in Matlab) to find eigenvalues of \eqref{eq:Delta_pert_3par} close to a target $\iu k_0$.  We can apply this several times using different targets $k_0$ and, thus, scan an interval $[k_a,k_b]$ for ZGV points $(k_*,\omega_*)$; for more details, see Section \ref{sec:alg} and \cite{ZGV_JASA_23}. In the next section, we will show how we can exploit the structure of the matrices \eqref{eq:threeeig_pert_1} and \eqref{eq:threeeig_pert_2} to find eigenvalues of \eqref{eq:Delta_pert_3par} close to a target $\iu k_0$ much more efficiently. This improvement enables the computation of ZGV points for much larger matrices than in the original algorithm from \cite{ZGV_JASA_23}.

We remark that it is not possible to apply a similar approach with a subspace iteration to the 3EP \eqref{eq:3ep2}
because this problem is singular.

\section{Exploiting the structure}\label{sec:impr}
When employing a subspace iterative method such as the Krylov-Schur method \cite{Stewart_Krylov_Schur} or the implicitly restarted Arnoldi method \cite{Lehoucq_Sorensen_IRA} for the solution of the GEP \eqref{eq:Delta_pert_3par_mu}, the computational bottleneck in each step is the solution of a linear system of the form
\begin{equation}\label{eq:linsisDelta1}
    (\Delta_1-\sigma\Delta_0)z=\Delta_0 y.
\end{equation}
Even for sparse matrices $L_0, L_1, L_2, M$, which we obtain using, e.g., the finite element method, this makes the computation slow already for modest matrix size $n$. By exploiting the structure of the matrices $\Delta_0$ and $\Delta_1$ in a similar way as in \cite{MP_SylvArnoldi}, we can solve the linear system \eqref{eq:linsisDelta1} much more efficiently.

First, we note the block structure
\begin{equation}
    \label{eq:delta_structure_02}
    \Delta_0= \left[\begin{matrix}
            G_1 & G_2 \cr
            G_2 & 0\end{matrix}\right],\quad
    \Delta_1= \left[\begin{matrix}
            -G_0 & 0 \cr
            0    & G_2\end{matrix}\right],
\end{equation}
where
\begin{align}\label{eq:G0G1G2}
    G_0 & = L_0 \otimes M -  M\otimes L_0,\nonumber             \\
    G_1 & = L_1 \otimes M - (1+\delta) M\otimes L_1,            \\
    G_2 & = L_2 \otimes M - (1+\delta)^2 M\otimes L_2\nonumber.
\end{align}
Introducing the block notation $z=\left[\begin{matrix}z_1 \cr z_2\end{matrix}\right]$ and $y=\left[\begin{matrix}y_1 \cr y_2\end{matrix}\right]$, we can rewrite \eqref{eq:linsisDelta1} as
\begin{equation}
    \label{eq:sisDelta1block}
    \left[\begin{matrix}
            -G_0-\sigma G_1 & -\sigma G_2 \cr
            -\sigma G_2     & G_2\end{matrix}\right]
    \left[\begin{matrix}z_1 \cr z_2\end{matrix}\right] =
    \left[\begin{matrix}
            G_1y_1 + G_2y_2 \cr
            G_2y_1\end{matrix}\right].
\end{equation}
If we add the second block row, multiplied by $\sigma$, to the first block row, we get an equivalent lower block triangular system
\begin{equation}
    \label{eq:sisDelta1block2}
    \left[\begin{matrix}
            -G_0-\sigma G_1-\sigma^2 G_2 & 0 \cr
            -\sigma G_2                  & G_2\end{matrix}\right]
    \left[\begin{matrix}z_1 \cr z_2\end{matrix}\right] =
    \left[\begin{matrix}
            (G_1+\sigma G_2)y_1 + G_2y_2 \cr
            G_2y_1\end{matrix}\right].
\end{equation}
Hence, we can compute $z_1$ using the first block row
\begin{equation}\label{eq:solve_z1}
    (G_0+\sigma G_1+\sigma^2 G_2)z_1=-(G_1+\sigma G_2)y_1-G_2y_2,
\end{equation}
and it follows from the second equation and nonsingularity of $G_2$ that $z_2=y_1+\sigma z_1$.

To solve \eqref{eq:solve_z1} efficiently, we transform it into a Sylvester equation using the equalities from Subsection~\ref{sub:Sylvester}.
First, let $w:=-(G_1+\sigma G_2)y_1-G_2y_2$. If $y_1={\rm vec}(Y_1)$ and $y_2={\rm vec}(Y_2)$, then it follows from
\eqref{eq:G0G1G2} that the right-hand side of \eqref{eq:solve_z1} is $w={\rm vec}(W)$, where
\[ W = MY_1(L_1+\sigma L_2)^\mup{T}-((1+\delta)L_1+(1+\delta)^2L_2)Y_1M^\mup{T}
    -MY_2L_2+(1+\delta)^2L_2Y_2M^\mup{T}.
\]
In a similar way, we get from \eqref{eq:G0G1G2} that
\[G_0+\sigma G_1+\sigma^2 G_2 = L(0)\otimes M - M\otimes L(\delta),\]
where $L(\delta) := L_0+(1+\delta)\sigma L_1+(1+\delta)^2\sigma^2 L_2$. Thus, \eqref{eq:solve_z1} is equivalent to
\begin{equation}\label{eq:MLDeltaZ}
    MZ_1L(0)^\mup{T}-L(\delta)Z_1M^\mup{T} = W,
\end{equation}
where $z_1={\rm vec}(Z_1)$. Since $M$ is nonsingular, we can write the above as a Sylvester equation
\begin{equation}\label{eq:MLDeltaZ_Sylvester}
    Z_1L(0)^\mup{T}M^{-\mup{T}}-M^{-1}L(\delta)Z_1 = \widetilde {W},
\end{equation}
where
$
    \widetilde W = M^{-1}WM^{-\mup{T}}
    =
    Y_1(L_1+\sigma L_2)^\mup{T}M^{-\mup{T}}-M^{-1}((1+\delta)L_1+(1+\delta)^2L_2)Y_1
    -Y_2L_2M^{-\mup{T}}+M^{-1}(1+\delta)^2L_2Y_2$.
As we can solve the above Sylvester equation in complexity ${\cal O}(n^3)$, this is also the complexity of solving \eqref{eq:linsisDelta1}.

Let us remark that we can exploit the structure of $\Delta_0$ and $\Delta_M$ to compute $\mu$ in \eqref{eq:Rayleigh_mu} in complexity ${\cal O}(n^3)$ as well. Namely, we have
\[    \Delta_M= \left[\begin{matrix}
            G_3 & G_4 \cr
            G_4 & G_5\end{matrix}\right]
\]
with
\begin{align*}
    G_3 & = -L_1 \otimes L_0 + (1+\delta)L_0\otimes L_1,\nonumber              \\
    G_4 & = (1+\delta)^2 L_0 \otimes L_2 - L_2\otimes L_0,                     \\
    G_5 & = -(1+\delta)L_2 \otimes L_1 + (1+\delta)^2 L_1\otimes L_2\nonumber.
\end{align*}
To obtain \eqref{eq:Rayleigh_mu}, we
compute matrices
\begin{align*}
    T_1 & = -L_0Z_1L_1^\mup{T} + (1+\delta)L_1Z_1L_0^\mup{T}+(1+\delta)^2L_2Z_2L_0^\mup{T}-L_0Z_2L_2^\mup{T},\nonumber    \\
    T_2 & = (1+\delta)^2L_2Z_1L_0^\mup{T}-L_0Z_1L_2^\mup{T} - (1+\delta) L_1Z_2L_2^\mup{T}+(1+\delta)^2L_2Z_2L_1^\mup{T}, \\
    T_3 & = MZ_1L_1^\mup{T}-(1+\delta)L_1Z_1M^\mup{T} + MZ_2L_2^\mup{T} - (1+\delta)^2 L_2Z_2M^\mup{T},                   \\
    T_4 & = MZ_1L_2^\mup{T} - (1+\delta)^2 L_2Z_1M^\mup{T}
\end{align*}
and compute $\mu$ as
\begin{equation}\label{eq:mu_Rayleigh}
    \mu=\frac{z_1^\mup{H}t_1+z_2^\mup{H}t_2}{z_1^\mup{H}t_3+z_2^\mup{H}t_4},
\end{equation}
where $t_i={\rm vec}(T_i)$ for $i=1,\ldots,4$. As this computation involves only multiplications by $n\times n$ matrices, its complexity is ${\cal O}(n^3)$.

\section{Algorithm}\label{sec:alg}
For large problems, we suggest applying the MFRD to provide good initial approximations, which we subsequently refine using the local convergent method presented next.

\subsection{Gauss-Newton method}\label{sec:Gauss-Newton}

If we introduce $\lambda = \iu k$ and $\mu=\omega^2$ similarly to \eqref{eq:waveguide_problem_discrete_pert}, then we know from Section \ref{sec:theory} that for a ZGV point of \eqref{eq:ZGV1}, we have to find $\lambda,\mu\in \CC$ and $u,z\in\CC^n$ such that
\begin{align}
    (\lambda^2 L_2+ \lambda L_1+L_0+\mu M)u          & =0\nonumber         \\
    z^\mup{H} (\lambda^2 L_2+ \lambda L_1+L_0+\mu M) & =0\nonumber         \\
    z^\mup{H}(2 \lambda L_2 + L_1 )u                 & =0\label{eq:sis_2D} \\
    (u^\mup{H} u - 1)/2                                & =0\nonumber         \\
    (z^\mup{H} z - 1)/2                                  & =0.\nonumber
\end{align}The number of equations exceeds the number of unknowns by one in (\ref{eq:sis_2D}); hence, this is an overdetermined nonlinear system. However, it is a \emph{zero-residual system} because if $(k,\omega)$ is a ZGV point, and $u$ and $z$ are the corresponding right and left eigenvector, then all equations in \eqref{eq:sis_2D} are satisfied.

For solving \eqref{eq:sis_2D} from a good initial approximation, we apply the Gauss-Newton method \cite{Deuflhard_2011, Nocedal_Wright_2006}. To overcome the obstacle that the third equation in \eqref{eq:sis_2D} is not complex differentiable in $z$, we define $y=\overline z$, where $\overline \bullet$ denotes the complex conjugate, and rewrite  \eqref{eq:sis_2D} as
\begin{equation}\label{eq:sis_2Dconj}
    F(u,y,\lambda,\mu):=
    \left[\begin{matrix}(\lambda^2 L_2+ \lambda L_1+L_0+\mu M) u\cr
            (\lambda^2 L_2+ \lambda L_1+L_0+\mu M)^\mup{T} y\cr
            y^\mup{T}(2\lambda L_2 + L_1 )u\cr
            (u^\mup{H} u - 1)/2\cr
            (y^\mup{H} y - 1)/2\end{matrix}\right]=0.
\end{equation}
Suppose that $(u_k,y_k,\lambda_k,\mu_k)$ is an approximation to the solution of \eqref{eq:sis_2Dconj}. If 
$F(u,y,\lambda,\mu)=0$ then also
$F(\alpha u,\beta y,\lambda,\mu)=0$ for arbitrary $\alpha,\beta\in\CC$ such that $|\alpha|=|\beta|=1$. Because of that,
the vectors $u$ and $y$ are not uniquely defined and, although the last two equations in \eqref{eq:sis_2Dconj} are not
complex differentiable, as explained in \cite{LuSu_NLAA}, we can 
obtain a correction $(\Delta u_k,\Delta y_k, \Delta\lambda_k,\Delta\mu_k)$ for the update
\[(u_{k+1},y_{k+1},\lambda_{k+1},\mu_{k+1})=(u_k,y_k,\lambda_k,\mu_k)+(\Delta u_k,\Delta y_k,\Delta\lambda_k,\Delta\mu_k)\]
as a solution of the $(2n+3)\times (2n+2)$ least squares problem 
$$J_F(u_k,y_k,\lambda_k,\mu_k)\Delta s_k=-F(u_k,y_k,\lambda_k,\mu_k),$$
where $\Delta s_k =
    \left[\begin{matrix}\Delta u_k^\mup{T} & \Delta y_k^\mup{T} & \Delta\lambda_k & \Delta\mu_k\end{matrix}\right]^\mup{T}
$, and the Jacobian $J_F(u_k,y_k,\lambda_k,\mu_k)$ is
\begin{equation}\label{eq:jacobian}
    \left[\begin{matrix}
            \lambda_k^2 L_2+ \lambda_k L_1+L_0+\mu_k M & 0                                                    & (2\lambda_k L_2 + L_1)u_k         & Mu_k \cr
            0                                          & (\lambda_k^2 L_2+ \lambda_k L_1+L_0+\mu_k M)^\mup{T} & (2\lambda_k L_2 + L_1)^\mup{T}y_k & M^\mup{T}y_k \cr
            y_k^\mup{T}(2\lambda_k L_2 + L_1 )         & u_k^\mup{T}(2\lambda_k L_2 + L_1 )^\mup{T}           & 2y_k^\mup{T}L_2u_k                & 0\cr
            u^\mup{H}                                  & 0                                                    & 0                                 & 0\cr
            0                                          & y^\mup{H}                                            & 0                                 & 0\end{matrix}\right].
\end{equation}

Besides an initial approximation $(k_0,\omega_0)$ for the ZGV point, the method requires initial approximation for
the right and left eigenvector as well. If we do not have them, then usually a good choice is to use a random vector from the space spanned by the right and left singular vectors that belong to a few of the smallest singular values of $W(k_0,\omega_0)$.

The Gauss-Newton method converges locally quadratically for a zero-residual problem if the Jacobian $J_F$ has full rank at the solution, see, e.g., \cite[Section 4.3.2]{Deuflhard_2011} or \cite[Section 10.4]{Nocedal_Wright_2006}.
We can show that the Jacobian $J_F(u_*,y_*,\lambda_*,\mu_*)$ has full rank at a generic ZGV point, where $k_*=-\iu\lambda_*$ is a double eigenvalue of \eqref{eq:QEP_Q} for $\omega=\sqrt{\mu_*}$, and $u_*$ and $\overline y_*$ are the corresponding right and left eigenvector. For the proof, see Lemma \ref{lem:jacobi_fullrank} in the appendix. If ZGV points exist where the multiplicity of $k_*$ is higher than two, then, at such points, we can expect a linear convergence.

Let us note that the above Gauss-Newton method also converges to points $(\lambda,\mu)$ where there exists a solution of \eqref{eq:sis_2D}, and these include the points where the dispersion curves cross. At such points, we can also expect a linear convergence. If the method has converged to a zero-residual solution, then we can verify if a computed point is a ZGV point by checking if $\omega_*$ is indeed a simple eigenvalue of $W(k_*,\omega)$.

We also remark that for a special case when $W(k,\omega)$ is Hermitian for real values $k$, the left eigenvector corresponding to a right eigenvector $u$ is $u^\mup{H}$, which happens, for example, in the case when the matrices $L_i$ are alternately symmetric/anti-symmetric and $M$ is symmetric positive definite. For such problems, we can use a more efficient Newton's method, see \cite[Sec. IV]{ZGV_JASA_23}.

\subsection{Scanning method}

The following algorithm uses the MFRD from Section \ref{sec:mfrd} to scan a wavenumber interval $[k_a,k_b]$ and
compute ZGV points $(k_*,\omega_*)$ such that $k\in[k_a,k_b]$. The matrices $\Delta_0$, $\Delta_1$ and $\Delta_M$ refer to \eqref{eq:threeeig_pert_1}, \eqref{eq:threeeig_pert_2} and \eqref{eq:threeeig_pert_M}, respectively, but we do not have to compute them explicitly.

    {
        \noindent\vrule height 0pt depth 0.5pt width \linewidth \\*
        \textbf{Algorithm~1: Scanning method for ZGV points} \\
        \textbf{Input:} $n\times n$ matrices $L_2, L_1, L_0, M$, interval $[k_a,k_b]$, default step size $\Delta k$ \\
        \textbf{Output:} ZGV points $(k_*, \omega_*)$\\*[-1.2ex]
        \vrule height 0pt depth 0.3pt width \linewidth \\[2pt]
        \begin{tabular}{ll}
            {\footnotesize \phantom{1}1:} & set $k_0=k_a$ \cr
            {\footnotesize \phantom{1}2:} & while $k_0<k_b$ \cr
            {\footnotesize \phantom{1}3:} & \hbox{}\quad find $m$ eigenvalues of $\Delta_1z = \lambda \Delta_0 z$ closest to $\lambda_0 = \iu k_0$\cr
            {\footnotesize \phantom{1}4:} & \hbox{}\quad for each $\lambda$ and eigenvector $z$\cr
            {\footnotesize \phantom{1}5:} & \hbox{}\quad\quad compute $\mu=z^\mup{H} \Delta_M z /z^\mup{H} \Delta_0 z$\cr
            {\footnotesize \phantom{1}6:} & \hbox{}\quad\quad if $|\textrm{Re}(\lambda)|$ and $|\textrm{Im}(\mu)|$ are both small then \cr
            {\footnotesize \phantom{1}7:} & \hbox{}\quad\quad\quad apply Gauss-Newton method to solve \eqref{eq:sis_2D} with initial guess $(\textrm{Im}(\lambda),\textrm{Re}(\mu))$\cr
            {\footnotesize \phantom{1}8:} & \hbox{}\quad\quad\quad if the method converged to $(\lambda_*,\mu_*)$, and \eqref{eq:cond_ZGV} holds then \cr
            {\footnotesize \phantom{1}9:} & \hbox{}\quad\quad\quad\quad add $(k_*,\omega_*) = (-\iu \lambda_*,\sqrt{\mu_*})$ to the list of ZGV points \cr
            {\footnotesize 10:}           & \hbox{}\quad\quad set $k_0=\max(k_0+\Delta k,0.95\cdot \max\{k_*\in \textrm{ZGV points}\})$\cr
        \end{tabular}\\*
        \vrule height 0pt depth 0.5pt width \linewidth \\
    }

In the following, we provide additional details about Algorithm 1.
\begin{itemize}
    \item In line 3, we can apply any subspace method, for instance {\tt eigs} in Matlab. Thereby, it is important that we do not generate the matrices $\Delta_0$ and $\Delta_1$ explicitly. Internally, the {\tt eigs} function iteratively solves  
    the linear system
          $(\Delta_1-\sigma\Delta_0)z=\Delta_0 y$, and it is possible to provide a pointer to a custom implementation thereof. We do so, thereby exploiting the relation with the Sylvester equation from Section \ref{sec:impr}.
    \item Alternatively, if $n$ is small enough, we can compute all eigenvalues of $\Delta_1z = \lambda \Delta_0 z$.
          This approach gives all ZGV points in just one run; hence, the scanning is not needed, and it could be more efficient than solving the related singular 3EP \eqref{eq:3ep2}.
    \item In line 5, we compute $\mu$ using \eqref{eq:mu_Rayleigh}, which avoids generating $\Delta_0$ and $\Delta_M$. Note that this expression
         might not return correct $\mu$ if $\lambda$ is a multiple eigenvalue of $\Delta_1z = \lambda \Delta_0 z$, e.g., when
         there exist different ZGV points with the same $k=-\iu \lambda$. This is very unlikely to appear, except for the trivial
         ZGV points at $\lambda=0$. This makes ZGV points with $k_*$ close to zero very difficult to compute.
         
    \item For an initial approximation for the left and right eigenvector, we take the left and right singular vector for
    the smallest singular value of $\lambda^2 L_2+ \lambda L_1+L_0+\mu M$.
    \item In line 10, we update the target $k_0$ in such a way that it is unlikely that the method will miss ZGV points in the interval. We assume that if the subspace iteration method in line 3 returns some approximations, it does not miss any of the closest ZGV points. The idea is to increase the target either for a default step $\Delta k$ 
    or use a larger step if some ZGV points were found in the last loop.
\end{itemize}
It is difficult to provide the best values for the parameters $k_a, k_b, \Delta k, m$, as they are problem-dependent. For a sensible choice when treating guided waves in plates, see the implementation in the package GEWtool \cite{kiefer_gewTool} and the numerical examples in Section~\ref{sec:num_exp}.

\section{Waveguide model}\label{sec:waveguide}
While the proposed approach can be applied, quite generally, to compute critical points on eigencurves of parameter-dependent quadratic eigenvalue problems, this work is motivated by a particular application, namely the modeling of elastic waves propagating along structures of constant cross-section, commonly referred to as waveguides. In this context, a finite-element discretization of the cross-section yields matrices with the properties discussed above. Hence, we will briefly summarize the formulation that has been deployed to obtain the matrices used in our numerical experiments.
Consider a waveguide of linearly elastic material and arbitrary cross-section~$\Gamma$ as depicted in Fig.~\ref{fig:geometry}a. Its mass density is denoted as $\rho$, and its 4$^\text{th}$-order stiffness tensor~$\ten{c}$ is of arbitrary anisotropy. In absence of external loads, the mechanical displacements $\tilde{\ten{u}}(x,y,z,t)$ are governed by the boundary-value problem
\begin{subequations}\label{eq:bvp_real}
    \begin{align}
        \nabla \cdot \ten{c} : \nabla \tilde{\ten{u}} - \rho \partial_t^2 \tilde{\ten{u}} & = \ten{0} \quad \text{in} \quad \Gamma \times \mathbb{R} \,,           \\
        \tilde{\ten{u}}                                                                   & = \ten{0} \quad \text{on} \quad \partial\Gamma_D \times \mathbb{R} \,, \\
        \dirvec{\mup{n}} \cdot \ten{c} : \nabla \tilde{\ten{u}}                           & = \ten{0} \quad \text{on} \quad \partial\Gamma_N \times \mathbb{R} \,.
    \end{align}
\end{subequations}
Therein, $\partial_t$ is the partial derivative with respect to time~$t$, and $\nabla = \dirvec{i} \partial_i$ is the Nabla operator.
$\partial\Gamma_D$ and $\partial\Gamma_N$ denote the parts of the cross-sectional boundary where Dirichlet and Neumann boundary conditions are imposed, respectively. Lastly,  $\dirvec{\mup{n}}$ is the outward unit normal vector.

\begin{figure}[tbh]
    \centering
\begin{tikzpicture}[
inner sep=1pt, outer sep=1pt,
scale=1,
]
\footnotesize
\node[below] at (-1.3cm, -1cm-1em) {(a) arbitrary 2d cross-section};

\path (-1, -1) coordinate (p1) -- 
(1, 1) coordinate (p2) -- 
(0, 2) coordinate (p22) -- 
(-3, 2) coordinate (p3) -- 
(-4, 0) coordinate (p4) -- 
cycle;

\draw [thick, rounded corners=.5cm, fill=color_plate] (p1) --  (p2) -- (p22) -- (p3) -- (p4) -- cycle;

\draw [draw=inferno3, very thick, rounded corners=.5cm] ($(p1)!0.3!(p2)$) -- node[sloped, inner sep=0pt, below, pos=.6, anchor=north west, minimum height=.8cm, minimum width=.3cm] (nvec){}  (p2) -- node[inner sep=0pt, pos=.6] (N){} (p22) -- ($(p22)!0.5!(p3)$);

\draw [draw=viridis3!80!green, very thick, rounded corners=.5cm] ($(p22)!0.5!(p3)$) -- (p3) -- (p4) -- node[inner sep=0pt, pos=.2] (S) {} (p1) -- ($(p1)!0.3!(p2)$);

\node[] at (-0.2,1) {$\Gamma$};
\node[] at (-1.2,-0.4) {$\ten{c}$, $\rho$};
\node[] (labelS) at ($(S) + (-170:1.2cm)$) {$\partial\Gamma_D$};
\draw[->] (labelS.east) .. controls +(-20:.4) and +(-120:.4) .. (S);
\node[] (labelN) at ($(N) + (0:1.2cm)$) {$\partial\Gamma_N$};
\draw[->] (labelN.west) .. controls +(170:.2) and +(40:.4) .. (N);

\draw[-Latex] (nvec.north west) coordinate (O) -- (nvec.south west) coordinate (B) node[below left] {$\dirvec{\mup{n}}$};
\draw pic [draw=black, very thin, angle radius=3mm, pic text=$.$, angle eccentricity=.5] {right angle = p1--O--B};

\coordinate (kVec) at (-1.6,0.7);
\draw[fill=white] (kVec) circle (2.5pt) node[above, inner sep=5pt, align=center]{$\ten{k}$};
\draw[fill=black] (kVec) circle (1pt);

\coordinate (Ocos) at (-4.6,.9); 
\draw[-latex, thick] (Ocos) -- +(1.5em, 0) node[below, inner sep=3pt, pos=1.1]{$\dirvec{y}$};
\draw[-latex, thick] (Ocos) -- +(0, 1.5em) node[left, inner sep=2pt, pos=1]{$\dirvec{z}$};
\draw[fill=white] (Ocos) circle (2.5pt) node[below left, inner sep=1pt, align=center]{$\dirvec{x}$};
\draw[fill=black] (Ocos) circle (1pt);




\end{tikzpicture}\hspace{1cm}%
\begin{tikzpicture}[>=stealth,
plate/.style={fill=color_plate, minimum width=\l, minimum height=\h, inner sep=0pt},
]
\footnotesize
\pgfmathsetmacro\h{1 cm}
\pgfmathsetmacro\l{4.7 cm}

\node[below] at (0, -1.2cm-1em) {(b) plate (1d cross-section)};

\node[plate] (plate) at (0, 0) {};
\draw[thin] (plate.north west) -- (plate.north east); 
\draw[thin] (plate.south west) -- (plate.south east); 

\coordinate (domainTop) at ([xshift=-1.2cm]plate.north);
\coordinate (domainBottom) at ([xshift=-1.2cm]plate.south);
\coordinate (domainCenter) at ($(domainTop)!0.5!(domainBottom)$);

\draw[thick] (domainTop) -- (domainBottom);
\fill[inferno3] (domainTop) circle (1.5pt);
\fill[viridis3!80!green] (domainBottom) circle (1.5pt);
\draw[->] (domainCenter)+(3ex,0) node[right] {$\Gamma$} -- (domainCenter);
\draw[<-, shorten <= 1.5pt] (domainTop) .. controls +(90:.3) and +(-180:.2) .. +(3ex,2ex) node[right]{$\partial\Gamma_N$};
\draw[<-, shorten <= 1.5pt] (domainBottom) .. controls +(-90:.3) and +(-180:.2) .. +(3ex,-2ex) node[right]{$\partial\Gamma_D$};

\draw[Bar-Bar] ([xshift=4pt]plate.north east) -- node[right]{$h$} ([xshift=4pt]plate.south east);

\node[left, yshift=0pt] at (plate.east) {$\ten{c}, \rho$};

\coordinate (kVec) at (0.2,0);
\draw[->, thick] (kVec) -- +(2em,0) node[right, inner sep=2pt]{$\ten{k}$};

\begin{scope}[shift={(-0.5*\l-2.5em, 0)}]
    \draw[-latex, fill opacity = 1] (0, 0) -- (1.5em, 0) node[below, inner sep=3pt, pos=1.1]{$\dirvec{x}$};
    \draw[-latex, fill opacity = 1] (0, 0) -- node[left, inner sep=2pt, pos=1]{$\dirvec{y}$} (0, 1.5em);
    \draw[fill=white] (0,0) circle (3pt) node[below left, inner sep=2pt, align=center]{$\dirvec{z}$};
    \draw[fill=black] (0,0) circle (1pt);
\end{scope}

\end{tikzpicture}
    \caption{Waveguide geometries. The waveguides extend infinitely along $\dirvec{x}$, which corresponds to the wave vector orientation. (a) Arbitrary two-dimensional cross-section. (b) A plate confines waves only in the one-dimensional thickness direction.
        $\Gamma$: cross-sectional domain, $\partial\Gamma_D$ Dirichlet boundary, $\partial\Gamma_N$: Neumann boundary, $\dirvec{\mup{n}}$: unit normal to the boundary, $\ten{c}$: stiffness tensor, $\rho$: density, $\ten{k}$: wave vector, $h$: thickness.}
    \label{fig:geometry}
\end{figure}
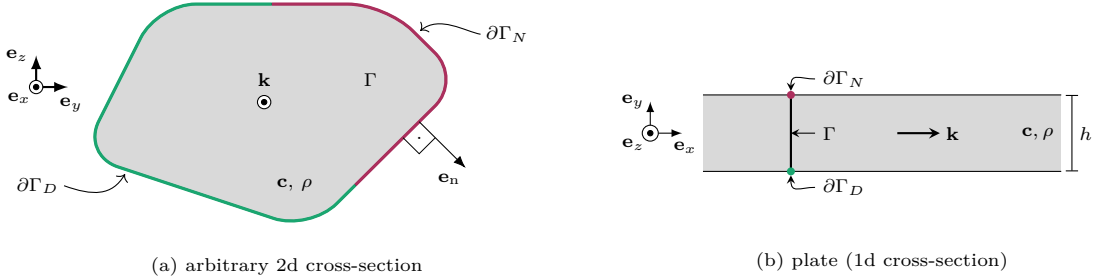

Due to the translational invariance in time $t$ and the axial coordinate~$x$, modal solutions are of the form $\tilde{\ten{u}}(x,y,z,t) = \ten{u}(k,y,z,\omega)\e^{\iu(k x - \omega t)}$. Inserting into (\ref{eq:bvp_real}), we obtain the waveguide problem
\begin{subequations}%
    \begin{align}%
        \label{eq:wp_motion}%
        \left[(\iu k)^2 \mathcal{L}_2 + \iu k \mathcal{L}_1 + \mathcal{L}_0 + \omega^2 \rho \ten{I} \right] \ten{u} & = \ten{0} \quad \text{in} \quad \Gamma \,,           \\
        \label{eq:wp_dirichlet}
        \ten{u}                                                                                                     & = \ten{0} \quad \text{on} \quad \partial\Gamma_D \,, \\
        \label{eq:wp_neumann}
        \left[ \iu k \mathcal{B}_1 + \mathcal{B}_0 \right] \ten{u}                                                  & = \ten{0} \quad \text{on} \quad \partial\Gamma_N \,,
    \end{align}\label{eq:waveguide_problem}%
\end{subequations}
which describes the plane harmonic guided wave solutions $(k, \omega, \ten{u})$ of interest. In the above, $\mathcal{L}_i$ and $\mathcal{B}_j$ are differential operators, which are explicitly given, using the 2$^\text{nd}$-order tensors $\ten{c}_{ij} = \dirvec{i} \cdot \ten{c} \cdot \dirvec{j}$, as
\begin{subequations}
    \begin{align}
        \mathcal{L}_2 & = \ten{c}_{xx} \,,                                                                                                \\
        \mathcal{L}_1 & = (\ten{c}_{xy} + \ten{c}_{yx}) \partial_y + (\ten{c}_{xz} + \ten{c}_{zx}) \partial_z \,,                         \\
        \mathcal{L}_0 & = \ten{c}_{yy} \partial_y^2 + (\ten{c}_{yz} + \ten{c}_{zy}) \partial_y \partial_z + \ten{c}_{zz} \partial_z^2 \,,
    \end{align}\label{eq:waveguide_operators_differential}%
\end{subequations}
and
\begin{subequations}
    \begin{align}
        \mathcal{B}_1 & = \ten{c}_{nx} \,,                                      \\
        \mathcal{B}_0 & = \ten{c}_{ny} \partial_y + \ten{c}_{nz} \partial_z \,.
    \end{align}\label{eq:boundary_operators_differential}%
\end{subequations}
Equations (\ref{eq:waveguide_problem})-(\ref{eq:boundary_operators_differential}) are also valid for the special case of the infinite plate as depicted in Fig.~\ref{fig:geometry}b (the reader may refer to \cite{ZGV_JASA_23} for a succinct derivation). A plane strain field should be assumed in this case, i.e., all terms in $\partial_z$ vanish while the equations remain otherwise unaffected. The waveguide's cross-section can thereby be modeled by a one-dimensional discretization, as the displacements do not depend on $z$. Furthermore, by expressing the problem in cylindrical coordinates, a similar formulation can be obtained for waves propagating along full or hollow cylinders. We also note that the above vector-field problem can be reduced to the scalar wave equation representing waves in a fluid medium, i.e., scalar acoustics.

The waveguide problem in (\ref{eq:waveguide_problem}) represents a differential eigenvalue problem, which we discretize using a Galerkin finite-element procedure~\cite{kauselSemianalyticHyperelementLayered1977,Huang1984b,Gavric1995a,Hladky-Hennion1996,Gravenkamp2012}.\footnote{Specifically, we opt, in this work, for a particular type of high-order polynomial interpolation (sometimes referred to as spectral elements)  for one-dimensional cross-sections \cite{Gravenkamp2012,Gravenkamp2014} and for nonuniform rational B-splines (NURBS) \cite{Gravenkamp2016} to discretize complex 2D cross-sections. A review of various shape functions in the context of semi-analytical methods is given in \cite{Gravenkamp2019}.} This numerical formulation yields matrices $L_i, M \in \mathbb{R}^{n\times n}$ such that
\begin{align}
    \label{eq:waveguide_problem_discrete}
    W(k,\omega) u := \left[(\iu k)^2 L_2 + \iu k L_1 + L_0 + \omega^2 M \right] u & = 0
\end{align}
approximates the original problem (\ref{eq:wp_motion}) and respects the boundary conditions (\ref{eq:wp_dirichlet}) and (\ref{eq:wp_neumann}). Therein, $u$ is the vector of coefficients corresponding to the chosen discretization.
Note that, for a nondissipative material and real-valued parameters $k$ and $\omega$, the waveguide operator $W(k,\omega)$ is Hermitian, and, furthermore, $M$ is positive definite.

Elastic guided waves as described by (\ref{eq:waveguide_problem_discrete}) usually exhibit several ZGV points $(k_*, \omega_*)$. 
As the $\Delta$-matrices scale as $4n^2\times 4n^2$ (in eq. \eqref{eq:3ep2}) or $2n^2\times 2n^2$ (in eq. \eqref{eq:3ep_pert}), large waveguide problems quickly lead to prohibitively large computational demands, effectively rendering the methods from \cite{ZGV_JASA_23} unusable. There are mainly two crucial situations where this is the case: (i)~plates and cylinders with many layers, and (ii)~waveguides of two-dimensional cross-section. In the following, we demonstrate that the method presented in Section \ref{sec:impr} is capable of computing ZGV points even for such complex structures. 

\section{Numerical experiments}\label{sec:num_exp}
All numerical experiments are performed on an Apple M1 Pro notebook with 32\,GB of memory. The regularization parameter of the MFRD method is chosen as $\delta = 10^{-2}$ for all computations. 

\subsection{Austenitic steel plate}\label{sub:steelplate}
An orthotropic austenitic steel plate exhibits many ZGV points. They were computed in \cite{ZGV_JASA_23} with an MFRD algorithm that uses {\tt eigs} from Matlab 
to compute eigenvalues of \eqref{eq:Delta_pert_3par} close to a target $\iu k_0$. The matrices $\Delta_1$ and $\Delta_0$ were computed explicitly and are represented as sparse matrices. The method in Matlab 
first computes an LU sparse factorization, which is then used to solve \eqref{eq:linsisDelta1} in each step of the method. The example given in the mentioned reference is of size $n=39$, which yields $\Delta_i$-matrices of size $3042\times 3042$. The computational time with the old method was 12\,s.

Instead, we can use the approach proposed in Section~\ref{sec:impr} to solve \eqref{eq:linsisDelta1} without explicitly constructing the matrices $\Delta_1$ and $\Delta_0$. In the initial phase, we compute the Schur decompositions of $L(0)^\mup{T}M^{-\mup{T}}$ and $M^{-1}L(\delta)$ from \eqref{eq:MLDeltaZ_Sylvester} and then use them to solve \eqref{eq:MLDeltaZ_Sylvester} and, thus, obtain the solution of \eqref{eq:linsisDelta1}. In this way, we never use matrices larger than $n\times n$. Applying this procedure to the numerical example of \cite{ZGV_JASA_23} with $n=39$ finds all 18 ZGV points in 0.5\,s, which is more than twenty times faster than with the previous method. Note that the computational times are difficult to compare, as the strategy to update the target wavenumber also changed. More importantly, our new procedure scales favorably with the problem size, which is demonstrated by the following examples. 

\subsection{Fluid-filled pipe}\label{sub:pipe}
The following example, taken from Cui et al.~\cite{cuiBackwardWavesDouble2016}, consists of a water-filled steel pipe with a wall thickness of $h = 0.5\,\mathrm{mm}$ and an inner radius of 9.5\,mm, see Fig.~\ref{fig:pipeGeometry}.
The steel pipe is characterized by shear and longitudinal wave speeds of $c_\mathrm{t} = 3200\,\frac{\text{m}}{\text{s}}$, $c_\ell = 5900\,\frac{\text{m}}{\text{s}}$ and a mass density of $7900\,\frac{\text{kg}}{\text{m}^3}$. The water inside the pipe has a mass density of $1000\,\frac{\text{kg}}{\text{m}^3}$ and a bulk wave speed of $1500\,\frac{\text{m}}{\text{s}}$.
As the fluid domain is relatively large compared to the bulk wavelength in water, this problem requires a considerable number of degrees of freedom. Specifically, we used one element of 7$^\text{th}$ order (eight nodes) to discretize the pipe wall, while the fluid domain required a polynomial degree of 140 to yield accurate results within the selected frequency range.\footnote{While such large polynomial degrees are generally uncommon in the Finite Element Method, they have been found to be remarkably efficient in this particular context of waveguide modeling. This is because the bottleneck in the computation of the dispersion relations is the (complete) solution of an eigenvalue problem. The costs for this solution for moderate matrix sizes depend mainly on the matrix size rather than its sparsity. Hence, in contrast to most finite element applications, which require mainly the solution of linear systems of equations, it is, in this case, desired to obtain small matrices, even if they are dense. The advantage of such large element orders was described in \cite{Gravenkamp2012} and discussed in more detail in \cite{Gravenkamp2014}.}
In the pipe wall, we assume displacements $\ten{u}(r)\e^{\iu (k x + n_\varphi \varphi + \omega t)}$ of integer circumferential order $n_\varphi$ (similarly for the acoustic pressure in the fluid). This enables us to discretize only a radial line as sketched in Fig.~\ref{fig:pipeGeometry}. As an additional challenge, the pressure-displacement formulation leads to non-Hermitian matrices, and a complex formulation of the Newton-iteration refinement as described in Section \ref{sec:Gauss-Newton} is required. For comparison with the literature, we choose $n_\varphi = 0$ and $\ten{u} = u_x \dirvec{x} + u_r \dirvec{r}$ to obtain the so-called longitudinal modes $\mathrm{L}(0,m)$~\cite{cuiBackwardWavesDouble2016}. Overall, this results in matrices $L_i$ and $M$ in (\ref{eq:waveguide_problem_discrete}) of size $157 \times 157$ and $\Delta$-matrices of size $49298 \times 49298$.
As demonstrated by Cui et al.~\cite{cuiBackwardWavesDouble2016}, multiple ZGV points are found in the frequency region close to the backward wave of the empty pipe, i.e., the curve with negative slope in Fig.~\ref{fig:pipeDispersion}. Using parameters $m = 8$, $k_a h = 0.01$, $k_b h = 2$, $\Delta k h = 0.05$, our algorithm locates all 15 ZGV points in 11\,s. The result is depicted in Fig.~\ref{fig:pipeDispersion}.

\begin{figure}[tbh]
    \centering\footnotesize
    \tikzsetnextfilename{pipe_geometry}
    \subfloat[\label{fig:pipeGeometry}]{\raisebox{10mm}{
\def\rb{3.5}
\def\ra{3}
\tikzset{layerStyle/.style={minimum width={\l cm}, minimum height={\h cm}, inner sep=0pt}, anchor=south}
\tikzset{%
    layer/.pic={%
        \node[layerStyle,pic actions] (lay) {};
        \draw[thin] (lay.north west) -- (lay.north east); 
        \draw[thin] (lay.south west) -- (lay.south east); 
    }
}
%
\begin{tikzpicture}[>=stealth,]
\footnotesize

\node[circle, draw=black, fill=black!10, inner sep=0pt, minimum size={\rb cm}, anchor=center] (pipeWall) at (0,0) {};
\node[circle, draw=black, fill=blue!7, inner sep=0pt, minimum size={\ra cm}, anchor=center] (innerFluid) at (0,0) {};

\draw[<-, shorten <= 1.5pt] (pipeWall.-160) .. controls +(-160:0.5) and +(0:0.2) .. +(-0.8cm,1.5ex) node[left]{pipe wall};
\node[anchor=north] (fluidLabel) at (0,-2ex) {fluid};

\draw[-, thick] (0,0) -- (pipeWall.70) coordinate[pos=0.5](domain);

\draw[<-, shorten <= 1.5pt] (domain) .. controls +(160:0.7) and +(-90:0.5) .. +(-2.7cm,1.2cm) node[above]{line of discretization $\Gamma$};

\begin{scope}[shift={(70:2)}]
    \draw[-latex, fill opacity = 1] (0, 0) -- (70:1.5em) node[pos=1, anchor=east, inner sep=2pt]{$\dirvec{r}$};
    \draw[-latex, fill opacity = 1] (0, 0) -- (160:1.5em) node[left, inner sep=1pt]{$\dirvec{\varphi}$};
    \draw[fill=white] (0,0) circle (3pt) node[below right, inner sep=3pt]{$\dirvec{x}$};
    \draw[fill=black] (0,0) circle (1pt);
\end{scope}

\end{tikzpicture}}\quad}
    \subfloat[\label{fig:pipeDispersion}]{\includegraphics[height=6cm]{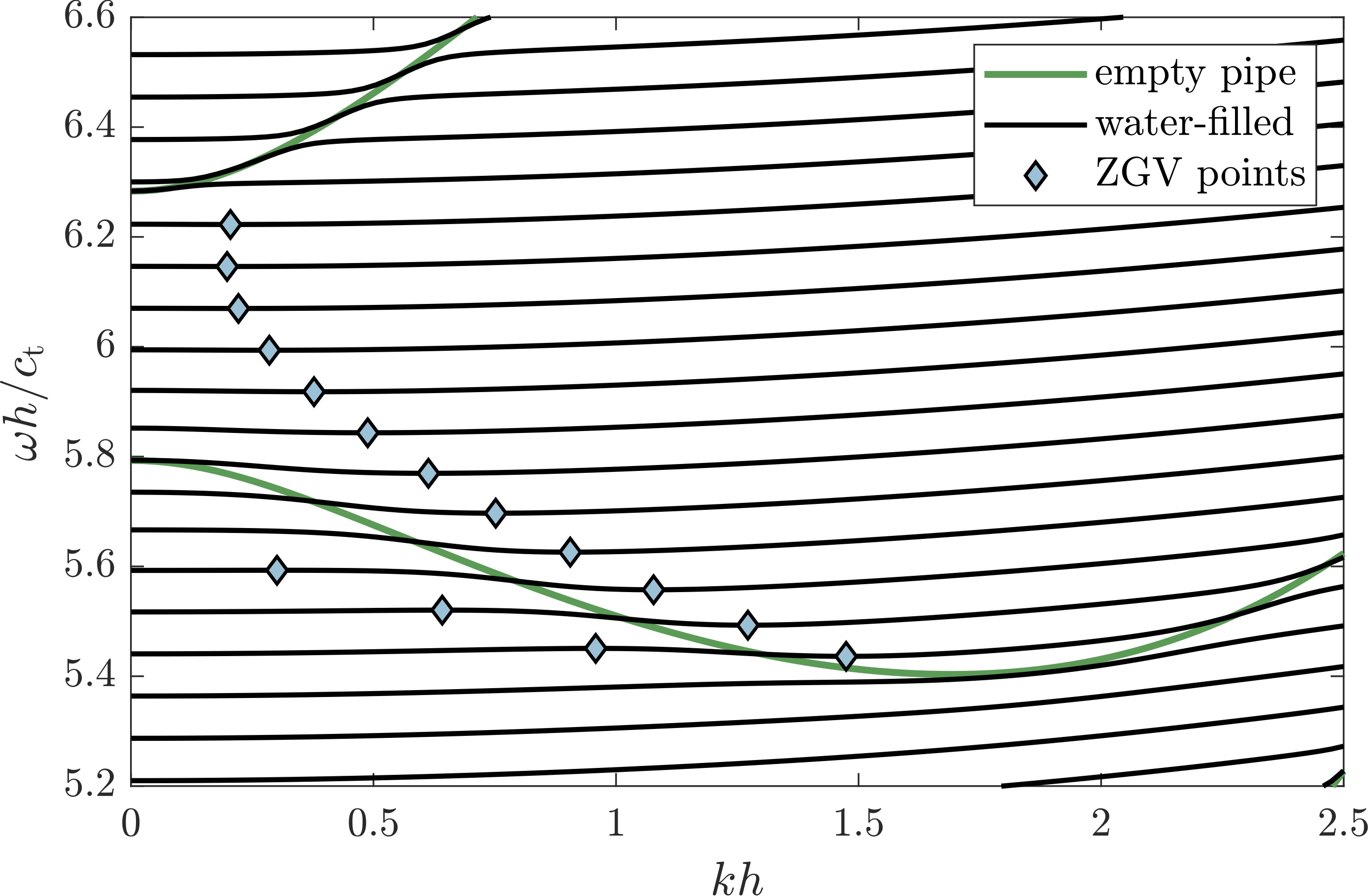}}
    \caption{Longitudinal waves in a water-filled steel pipe. (a)~Geometry: cross-section of the water-filled circular steel pipe of inner radius 9.5\,mm and outer radius 10\,mm. (b)~Dispersion curves of longitudinal modes, i.e., $u_x$-$u_r$-polarized waves. The dispersion curves of the empty pipe are shown for comparison. The fluid-filled pipe exhibits 15 ZGV points close to the backward wave of the free pipe.}
    \label{fig:pipe}
\end{figure}
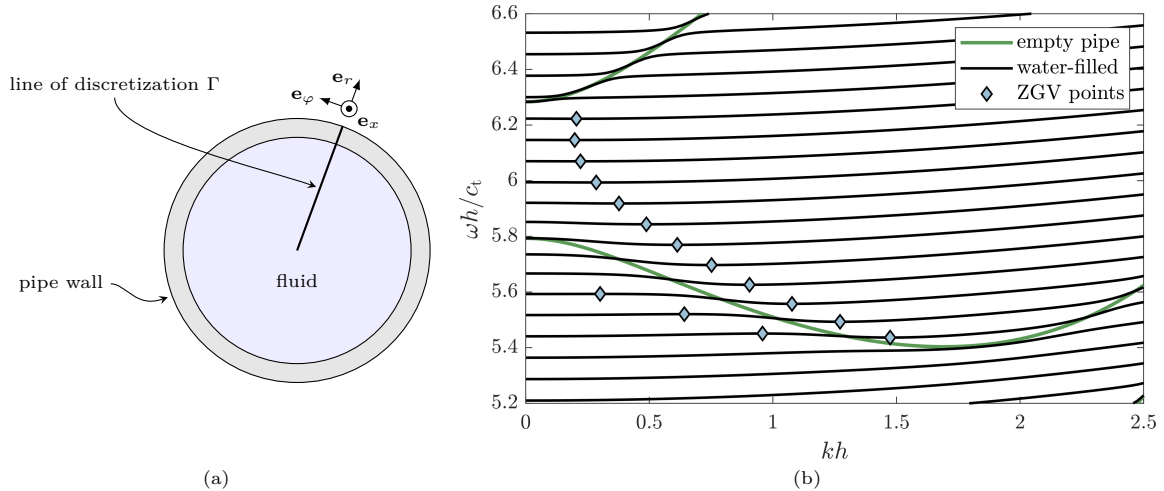

\subsection{Composite plate}\label{sub:composite_plate}
A plate consisting of many layers requires a large number of degrees of freedom to describe guided waves since each layer is represented by at least one finite element. We use the proposed method to compute ZGV points in a composite plate consisting of 400 layers with a total thickness of $h = \text{50\,mm}$. Such materials are used in the aerospace industry, and this particular example is taken from \cite{huber_classification_2018}. The plate is composed of a symmetric layup of a T800/913 carbon fiber reinforced polymer (CFRP) as depicted in Fig.~\ref{fig:compositeGeometry}. The mass density is given as $\rho = 1550\,\mathrm{kg/m^3}$, and the stiffness in Voigt notation reads
\begin{equation}
    \mathbf{C} =
    \begin{bmatrix}
        154 & 3.7                             & 3.7 & 0    & 0 & 0   \\
            & 9.5                             & 5.2 & 0    & 0 & 0   \\
            &                                 & 9.5 & 0    & 0 & 0   \\
            &                                 &     & 2.15 & 0 & 0   \\
            & \multicolumn{2}{c}{\text{sym.}} &     & 4.2  & 0       \\
            &                                 &     &      &   & 4.2
    \end{bmatrix} \mathrm{GPa} \,. \nonumber
\end{equation}
In order to consistently nondimensionalize the results, we define the smallest shear wave velocity as $c_\mathrm{t} = \sqrt{C_{44}/\rho}$.
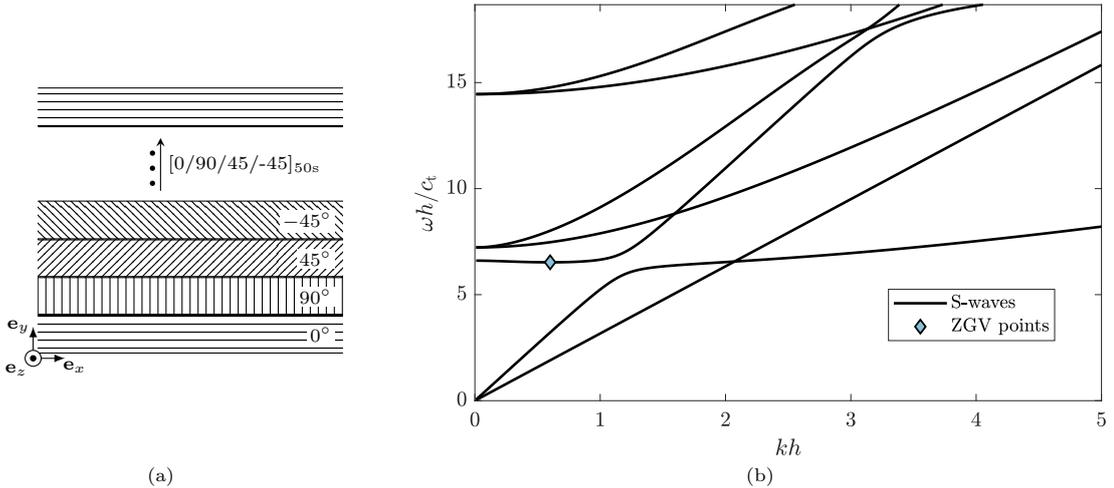
\begin{figure}[tbh]
    \centering\footnotesize
    \subfloat[\label{fig:compositeGeometry}]{\raisebox{1cm}{
\def\h{0.5}
\def\l{4}
\tikzset{layerStyle/.style={minimum width={\l cm}, minimum height={\h cm}, inner sep=0pt}, anchor=south}
\tikzset{%
    layer/.pic={%
        \node[layerStyle,pic actions] (lay) {};
        \draw[thin] (lay.north west) -- (lay.north east); 
        \draw[thin] (lay.south west) -- (lay.south east); 
    }
}
%
\begin{tikzpicture}[>=stealth,]
\footnotesize


\pic[fill=black!15, pattern=horizontal lines, pattern color=black] at (0,0)   {layer};
\node[left, fill=white, inner sep=1pt, outer sep=3pt] at (0.5*\l, 0.5*\h) {$0^\circ$};
\pic[fill=black!15, pattern=vertical lines, pattern color=black] at (0,\h)  {layer};
\node[left, fill=white, inner sep=1pt, outer sep=3pt] at (0.5*\l, 1.5*\h) {$90^\circ$};
\pic[fill=black!15, pattern=north east lines, pattern color=black] at (0,2*\h)  {layer};
\node[left, fill=white, inner sep=1pt, outer sep=3pt] at (0.5*\l, 2.5*\h) {$45^\circ$};
\pic[fill=black!15, pattern=north west lines, pattern color=black] at (0,3*\h)  {layer};
\node[left, fill=white, inner sep=1pt, outer sep=3pt] at (0.5*\l, 3.5*\h) {$-45^\circ$};
\pic[fill=black!15, pattern=horizontal lines, pattern color=black] at (0,6*\h)  {layer};

\path (-0.5cm,4.5*\h) -- (-0.5cm,5.5*\h)
  foreach \t in {0, 0.4, ..., 1} {
    pic [pos=\t] {code={\fill circle [radius=1pt];}}
  };
\draw[->] (-0.5cm+1ex,4.3*\h cm) -- node[right]{[0/90/45/-45]$_{\mathrm{50s}}$} (-0.5cm+1ex,5.7*\h cm);

\begin{scope}[shift={(-0.5*\l cm - 2pt, -2pt)}]
    \draw[-latex, fill opacity = 1] (0, 0) -- (1.5em, 0) node[pos=1.3, anchor=north, inner sep=1pt]{$\dirvec{x}$};
    \draw[-latex, fill opacity = 1] (0, 0) -- (0, 1.5em) node[left, inner sep=1pt]{$\dirvec{y}$};
    \draw[fill=white] (0,0) circle (3pt) node[below left, inner sep=2pt]{$\dirvec{z}$};
    \draw[fill=black] (0,0) circle (1pt);
\end{scope}

\end{tikzpicture}}}\hspace{3em}%
    \subfloat[\label{fig:compositeDispersion}]{\includegraphics[height=6cm]{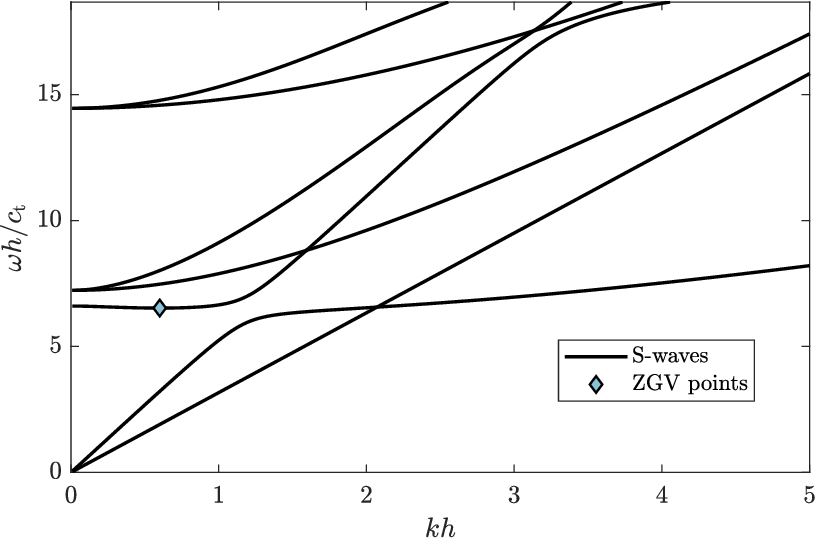}}
    \caption{Composite plate of symmetric layup [0/90/45/-45]$_{\mathrm{50s}}$. (a) Geometry: the fibers in each layer are oriented at angle $\theta$ as indicated; thereby, $\theta$ represents the angular coordinate rotating around $\dirvec{y}$ and measured from $\dirvec{x}$. All 400 layers are of the same thickness and amount to a total of 50\,mm. (b) Dispersion curves for wave vectors $\ten{k} = k \dirvec{x}$, i.e., $\theta = \mathrm{0}^\circ$. One point exists where the axial group velocity component $\partial \omega/\partial k$ vanishes, and it is marked in the plot.}
    \label{fig:composite_plate}
\end{figure}
For the frequency range of interest, it is sufficient to discretize each layer with one linear finite element. Symmetric and anti-symmetric waves decouple, and we consider symmetric waves only. This is achieved by representing one-half of the geometry and fixing the $u_y$ displacement at the center node. Lamb and shear-horizontal polarizations are coupled due to anisotropy. This requires modeling all three displacement components in the equations of motion (\ref{eq:waveguide_problem}). Proceeding as described previously yields the matrices $L_i$ and $M$ from (\ref{eq:waveguide_problem_discrete}) of size $602 \times 602$. Note that the corresponding $\Delta_i$-matrices are of size $724\,808 \times 724\,808$, which is considerable.

The dispersion curves corresponding to propagation along $\dirvec{x}$ are shown in Fig.~\ref{fig:composite_plate}b. There exists a point where $\partial \omega/\partial k = 0$, and it is marked therein. Using the parameters $m = 8$, $k_a h = 0.2$, $k_b h = 2$, $\Delta k h = 0.1$, our algorithm was able to successfully locate it in 43\,s. 
It is important to remark that the group velocity is a vector parallel to the $xz$\nobreakdash-plane. For anisotropic plates, it is not necessarily collinear to the wave vector~$\ten{k}$. The derivative $\partial \omega/\partial k$ represents the group velocity component along the wave vector $\ten{k}$. Since our numerical methods compute points such that $\partial \omega/\partial k = 0$, we find the waves whose group velocity is orthogonal to the wave vector or vanishes altogether. This was exploited in \cite{kiefer_beating_2023} to find waves with a power flux transverse to their wave vector.
As a side note, we also remark that the dispersion curves in Fig.~\ref{fig:composite_plate}b do not exhibit crossings. Instead, they get very close and then veer apart; see \cite{gravenkamp_notes_2023} for details on this phenomenon.

\subsection{Rail}\label{sub:rail}
In this numerical experiment, we compute ZGV points of a relatively complex three-dimensional structure, namely a rail with the cross-section depicted in Fig.~\ref{fig:railGeometry}.
Rails are typical examples of guided wave propagation in three-dimensional structures, and their dynamic properties are often investigated due to the relevance of acoustic emission and ultrasonic material testing, see, e.g., \cite{Gavric1995a, Hayashi2003a, Thompson1993, Zhang2016} and the references therein. This particular geometry has been studied in \cite{Gravenkamp2016}, where dispersion curves have already been computed based on the semi-analytical formulation outlined in Section~\ref{sec:waveguide}. Instead of conventional finite elements, the cross-section is discretized by means of non-uniform rational B-splines (NURBS), which allow for the exact description of this shape without introducing geometry approximation errors. Furthermore, NURBS are very robust at high frequencies. However, for the discussion in this paper, these differences are of lesser significance, as the obtained matrices possess the same relevant properties compared to using high-order polynomials.
For computing the ZGV points, we use the discretization suggested in \cite{Gravenkamp2016} for computing dispersion curves for the first nine modes up to a frequency of 10\,kHz. Specifically, the interpolation relies on the 30 patches shown in Fig.~\ref{fig:railGeometry}, each of them locally refined using NURBS of the third order, resulting in matrices of size $1020\times 1020$. For clarity, the figure only includes the minimal number of control points required to describe the geometry. A simple isotropic linearly elastic material is assumed with a Poisson's ratio of $\nu = 0.2$. To nondimensionalize the results for consistency with the other examples, we define $h = 172\,\mathrm{mm}$ as the height of the rail, i.e., the largest extent in the $y$-direction.
The dispersion curves are displayed in Fig.~\ref{fig:railDispersion}, together with the two ZGV points found within the selected frequency range. Using $m = 6$, $k_a h = 0.1$, $k_b h = 2$, $\Delta k h = 0.2$, Algorithm~1 locates the two ZGV points in 545\,s (9\,min).

\begin{figure}\centering
    \subfloat[\label{fig:railGeometry}]{\includegraphics[height=5.9cm]{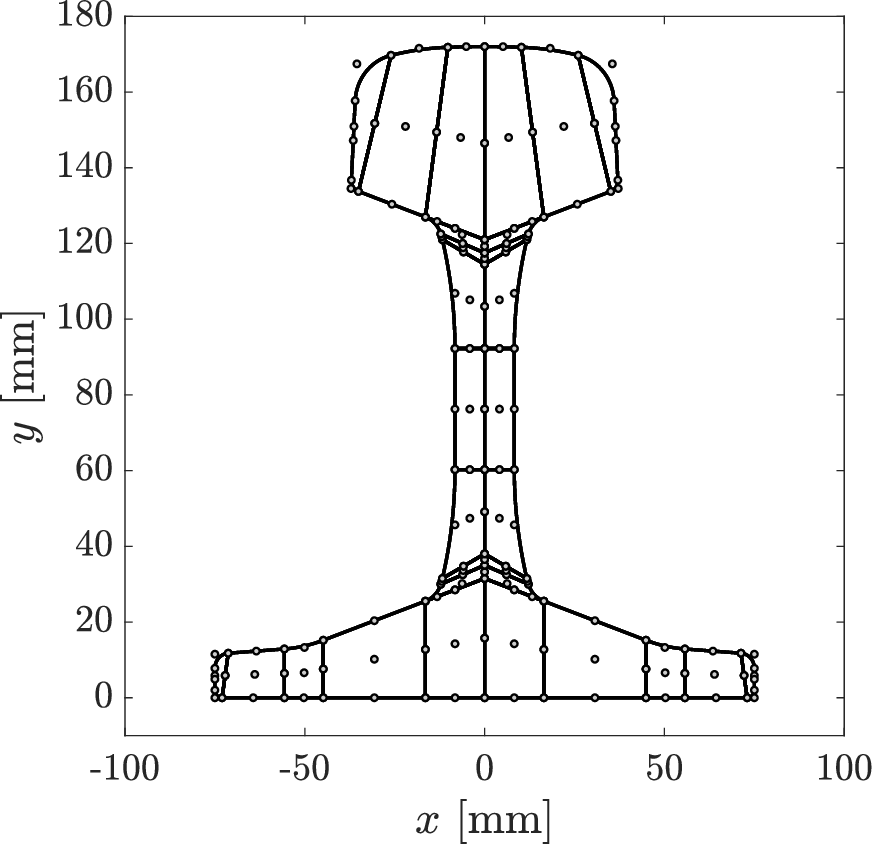}}\hfill
    \subfloat[\label{fig:railDispersion}]{\includegraphics[height=5.9cm]{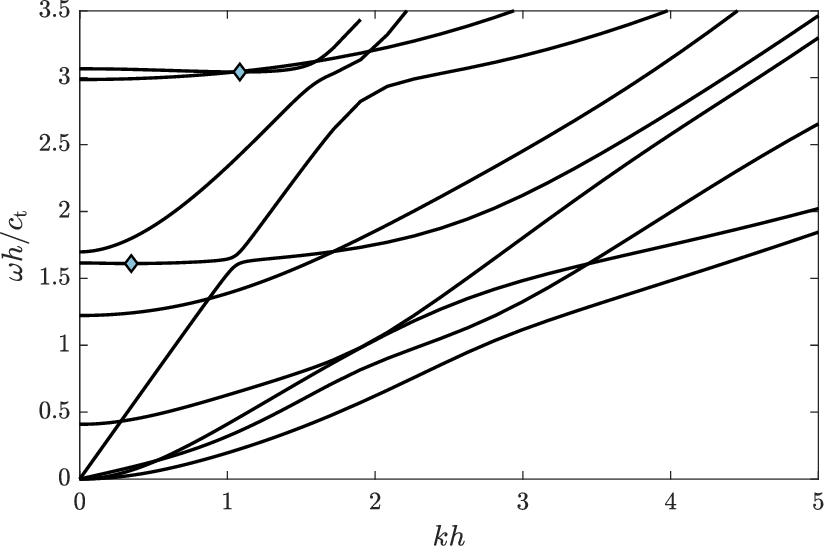}}
    \caption{Wave propagation along a rail. (a): Discretization of the cross-section, showing a division into 30 patches and the control points for describing the contour. (b): Dispersion curves of all propagating modes and the two ZGV points found within the chosen frequency range.\label{fig:rail}}
\end{figure}

\section{Possible generalizations}\label{sec:generalization}

It is possible to define ZGV points for similar two-parameter eigenvalue problems. For example, in \cite{BP_ZGV}, the critical points of dispersion curves for the eigenvalue problem $(A+\lambda B+\mu C)x=0$ are discussed. The obtained numerical methods for this problem are very similar, and we can apply the MFRD and the Gauss-Newton method, suitably modified to the structure of the eigenvalue problem. Specifically, we need to modify the approach in Section~\ref{sec:impr} so that we can use a subspace method in the MFRD without explicitly constructing large $\Delta$ matrices for the related MEP. In a similar way, it would be possible to define a ZGV point and extend the theory and numerical methods for a parameter-dependent polynomial eigenvalue problem of the form
\begin{equation}\label{eq:ZGVd}
    P(k,\omega)u:=\big((\iu k)^d L_d+ (\iu k)^{d-1} L_{d-1} + \cdots + \iu k L_1+L_0+\omega^2 M\big)\,u=0,
\end{equation}
or for a nonlinear parameter-dependent eigenvalue problem $N(k,\omega)u=0$, where $N: \CC^2\to \CC^{n\times n}$.

\section{Conclusion}\label{sec:conclusion}
The improved approach enables us to tackle significantly larger problems and compute more accurate solutions in cases where it was previously either impossible to construct the matrices $\Delta_0,\Delta_1,\Delta_M$ explicitly, or the computation was unfeasibly slow.
For even larger $n$, the improved algorithm also eventually reaches its limits due to the considerable memory requirements. Namely, vectors that span the search subspace in the subspace iteration method for the eigenvalue problem \eqref{eq:threeeig_pert_1} are of size $2n^2$. If the matrices $L_i$ and $M$ are large, the memory required for saving a sufficient number of vectors of this size can become prohibitively large.
A possible solution that could extend the approach to large and sparse matrices $L_i$ and $M$ would be to generalize the subspace iteration methods from \cite{MP_SylvArnoldi}, which exploit the low-rank format of the vectors and work only on vectors of size $n$.
\medskip

\noindent\textbf{Funding}\quad {Funding was provided by the
French Government (Program “Investissements d’Avenir”, Reference No. ANR-10-LABX-24) and Javna agencija za znanstvenoraziskovalno in inovacijsko dejavnost Republike Slovenije (Grant Nos. N1-0154 and P1-0194).}

\section*{Declarations}

\noindent\textbf{Conflict of interests}\quad Not applicable.

\bibliographystyle{abbrv} 
\bibliography{refs_BP, refs_MEP,refs_DK,ZGV_Sylvester_HG_2} 

\appendix

\section{Quadratic convergence of the Gauss-Newton method}

\begin{lemma}\label{lem:jacobi_fullrank}
    Let
    $\xi\in\Omega$ be an eigenvalue of algebraic multiplicity two and geometric multiplicity one of
    a nonlinear eigenvalue problem $N(\lambda)u=0$, where $N:\Omega\to\CC^{n\times n}$ is holomorphic on a domain $\Omega\subseteq\CC$.
    Let nonzero vectors $u,z,s,p\in\CC^n$ be, respectively, the right and left eigenvector and the right and left root vector of height two such that
    \begin{subequations}
        \begin{align}
            N(\xi)u & = 0,\label{eq:yBs_tBx:a}\\
            N(\xi)s+N'(\xi)u &=0,\label{eq:yBs_tBx:b}\\
            z^\mup{H}N(\xi) &=0,\label{eq:yBs_tBx:c}\\
            p^\mup{H}N(\xi)+z^\mup{H}N'(\xi) &=0.\label{eq:yBs_tBx:d}
        \end{align}\label{eq:yBs_tBx}%
    \end{subequations}
    Then \[z^\mup{H}N'(\xi) s + p^\mup{H}N'(\xi)u + z^\mup{H}N''(\xi)u\ne 0.\]
\end{lemma}
\begin{proof} If we multiply \eqref{eq:yBs_tBx:b} by $p^\mup{H}$ from the left
    and \eqref{eq:yBs_tBx:d} by $s$ from the right, we observe the equality $z^\mup{H}N'(\xi)s=p^\mup{H}N'(\xi)u$.
    Since it follows from \cite[Thm.~1.6.5]{MennickenMoller}, see also \cite[Thm.~2.5]{GuttelTisseur_NEP},
    that \[z^\mup{H}N'(\xi) s + z^\mup{H}\frac{N''(\xi)}{2}u\ne 0,\]
    this completes the proof.
\end{proof}

\begin{lemma}\label{lem:root_eigv}
    Let $(k_*,\omega_*)$ be a ZGV point of \eqref{eq:ZGV1}, such that
    the algebraic multiplicity of $k_*$ as an eigenvalue of the QEP \eqref{eq:QEP_Q} is two, and the geometric multiplicity is one.
    Let $u_*$ and $z_*$ be the corresponding right and left eigenvectors, $y_*=\overline z_*$, and let $a$ and $b$ be such vectors
    that $a^\mup{H}u_*=1$ and $b^\mup{H}y_*=1$.
    The Jacobian $J_F(u_*,y_*,\lambda_*,\mu_*)$, given in \eqref{eq:jacobian},
    where $\lambda_*=\iu k_*$ and $\mu_*=\omega_*^2$, has full rank.
\end{lemma}

\begin{proof}
    Suppose that the Jacobian $J_F(u_*,y_*,\lambda_*,\mu_*)$ is rank deficient. Then, there exist vectors
    $s,t$ and scalars $\alpha,\beta$, not all being zero, such that
    \[J_F(u_*,y_*,\lambda_*,\mu_*)\left[\begin{matrix}
                s^\mup{T} & t^\mup{T} & \alpha & \beta
            \end{matrix}\right]^\mup{T}=0.
    \]
    Then
     \begin{subequations}
     \begin{align}
        (\lambda_*^2 L_2+\lambda_* L_1+L_0+\mu_* M)s +\alpha (2\lambda_* L_2+L_1)u_* + \beta Mu_*                 & =0,\label{eq:dokaz1} \\
        (\lambda_*^2 L_2+\lambda_* L_1+L_0+\mu_* M)^\mup{T}t+\alpha (2\lambda_* L_2+L_1) y_* + \beta M^\mup{T}y_* & =0,\label{eq:dokaz2} \\
        y_*^\mup{T}(2\lambda_* L_2+L_1)s + u_0^\mup{T}(2\lambda_* L_2+L_1)^\mup{T}t+2\alpha y_*^\mup{T}L_2u_0     & =0,\label{eq:dokaz3} \\
        u_*^\mup{H}s                                                                                                & =0,\label{eq:dokaz4} \\
        y_*^\mup{H}t                                                                                                & =0\label{eq:dokaz5}.
    \end{align}
    \end{subequations}

    First, we show that $\beta=0$. If we multiply \eqref{eq:dokaz1} by $z_*^\mup{H}$, then it follows that $\beta=0$ because $z_*^\mup{H}(2\lambda_* L_2+L_1)u_0=0$ due to a ZGV point and $z_*^\mup{H}Mu_*\ne 0$ because we require that $\omega_*$ is a simple eigenvalue of $W(k_*,\omega)$.

    If $\alpha\ne 0$, then it follows from \eqref{eq:dokaz1} and \eqref{eq:dokaz2} that $(1/\alpha)s$ and $(1/\alpha)\overline t$ are left and right root vectors of height 2 of the QEP $N(\lambda)u:=(\lambda^2 L_2 + \lambda L_1 + L_0 + \mu_* M) u = 0$ for the eigenvalue $\lambda_*$. But then it follows from Lemma~\ref{lem:root_eigv} that \eqref{eq:dokaz3} is not zero. Therefore, $\alpha=0$.

    Since $\alpha= \beta=0$ it follows from \eqref{eq:dokaz1} that $s=\gamma u_*$ for a scalar $\gamma$ and then $s=0$
    because of \eqref{eq:dokaz4}. In a similar way we get from \eqref{eq:dokaz2} and \eqref{eq:dokaz5} that
    $t=0$. This shows that the kernel of $J_F(u_*,y_*,\lambda_*,\mu_*)$ is trivial.
\end{proof}

\end{document}